\documentclass[a4paper, 12pt,reqno]{amsart}
\usepackage{amsmath,amssymb,amsthm,enumerate}
\usepackage{textcomp}
\allowdisplaybreaks
\theoremstyle{plain}
\newtheorem{theorem}{Theorem}
\newtheorem*{theorem*}{Theorem}

\newtheorem*{lemma*}{Lemma}
\theoremstyle{definition}

\newtheorem*{definition*}{Definition}
\theoremstyle{remark}

\newtheorem*{remark*}{Remark}

\newtheorem*{statement*}{Statement}

\begin{document}
\title[The Engel--Minkowski function]{The Engel--Minkowski question-mark function}

\author{Symon Serbenyuk}



\maketitle
\text{\emph{simon6@ukr.net}}\\
\text{\emph{Kharkiv National University of Internal Affairs,}}\\
\text{\emph{L.~Landau avenue, 27, Kharkiv, 61080, Ukraine}}
\begin{abstract}

The present article deals with properties of a certain function of the Minkowski type with arguments defined by Engel series.  Differential, integral, and other properties of the function were considered.  
\end{abstract}

\section{Introduction}

\emph{An Engel series} is a series of the form
\begin{equation}
\label{eq: Engel series}
\frac{1}{q_1}+\frac{1}{q_1q_2}+\dots + \frac{1}{q_1q_2\cdots q_k}+\dots , 
\end{equation}
where $(q_k)$ is a fixed sequence of positive integers such that the condition $2\le q_k\le q_{k+1}$  holds for all $k\in\mathbb N$.

As noted in  \cite{ShW2021}, ``these infinite series expansions play an important role in the arithmetic and geometric natures of real numbers and are also closely connected with dynamical systems, probability theory and fractal geometry...". One can remark that various expansions of real numbers (\cite{{preprint1-2018}, {Symon2023-numeral-systems}}, etc.) are useful for modelling pathological (see \cite{Wikipedia-pathology}) mathematical objects with complicated local structure (the Moran sets (see \cite{Bunde1994, Falconer1997, Falconer2004,  Mandelbrot1999, Moran1946, HRW2000, sets1, sets2, sets}, etc.,  and references therein), singular (for example, \cite{{Salem1943}, {Zamfirescu1981}, {Minkowski}, {S.Serbenyuk 2017}, Symon2023,Symon2024}), non-differentiable (for example, see \cite{{Bush1952}, {Serbenyuk-2016}}, etc.), and nowhere monotonic  \cite{Symon2017, Symon2019, Symon2023, Symon2024} functions), which have applied character for modelling   real objects and  processes in  economics, physics, and technology, etc.,  as well as with different areas of mathematics (for example, see~\cite{BK2000, ACFS2011, Kruppel2009, OSS1995,    Symon21, Symon21-1, Sumi2009, Takayasu1984, TAS1993, Symon2021}). Such new objects  are also useful for discovering new mathematical relationships (for example, see \cite{ PVB2001, PKT2024, Salem1943} and references
therein).

As for sets of the Moran type, related  investigations of these sets, of multifractal analysis/formalism, and of  the refined multifractal formalism (for example, see  \cite{ XW2008, {DSM2021-EPhJ}, {Selmi2022-ASM}, {ALSW2024}, {ALSW24}, {ChLS2024}, {TESCh2024}} and references therein),  are one of the key areas in modern mathematics. Really,   ``the multifractal analysis was proven to be a very useful technique in the analysis of measures, both in theory and applications.  Also, the multifractal and  fractal analysis allows one to perform a certain classification of singular measures"\cite{Serbenyuk2024}. 

According to the described motivation, let us introduce a task to model  pathological functions of the Minkowski type in terms of various (positive and sign-variable) expansions  of real numbers. This article is the first article in the corresponding tuple of papers of the author. In this research, the main attention will be given to the Engel--Minkowski function. Such function can be useful for future investigations of number approximations and relationships between numeral systems. 

The classical Minkowski function was introduced in \cite{Minkowski} and was investigated by a number of researhers (for example, see \cite{PVB2001, {S.Serbenyuk 2017}}).

As noted in the paper~\cite{Symon2023-numeral-systems}, there exists the tendency to  investigate  Engel series incuding their positive and alternating forms (for example, see \cite{Engel1, Engel 3} and references in \cite{Engel2, Engel4}) are such examples. But in the last years, the main attention of researchers is given to certain sets defined in terms of these series.  
Really, for example, the papers \cite{LL2018, LW2003, ShW2020} deal with the Hausdorff dimension of some sets including sets of the full Hausdorff dimension, as well as \cite{FW2018, ShW2021, Wu2000} are devoted to  the growth speed of digits in Engel expansions including the Galambos conjecture (\cite{G1976}) and its generalizations.  Indeed, in \cite{ShW2021}, the growth speed of digits in Engel expansions from a multifractal point of view is investigated, and in  \cite{Wu2000}, it is proven the Galambos conjecture  on the fact that the Hausdorff dimension of the set
$$
\left\{x: x\in (0,1],  \lim_{n\to\infty}{q^{1/n} _n(x) \ne e}\right \}
$$
is  equal to $1$. A  survey on Engel series  is given more detail in \cite{ShW2021}.

\section{Auxiliary notions and definition of the main object}
  
Any real number $x\in (0,1]$ can be represented by series \eqref{eq: Engel series}.  In \cite{ShW2021}, it is noted that ``a real number is rational if and only if the sequence of its digits is ultimately periodic",  but according to  \cite{ERS1958},  rational numbers can be having a finite Engel expansion.

One can note that  the following inequalities hold for any $k\in \mathbb N$ (\cite[p.10]{ERS1958}):
\begin{equation}
\label{eq: Engel series1}
\sum^{k-1} _{l=1}{\prod^{l} _{j=1}{\frac{1}{q_j}}}+
 \frac{1}{q_1q_2\cdots q_{k-1}q_k} \le x<\sum^{k-1} _{l=1}{\prod^{l} _{j=1}{\frac{1}{q_j}}}+ \frac{1}{q_1q_2\cdots q_{k-1}(q_k-1)}. 
\end{equation}

A map of the form
$$
\sigma(x)=\sum^{\infty} _{k=2}{ \frac{1}{q_1q_2\cdots q_{k-1}q_k}}
$$
is  \emph{the shift operator} and 
$$
x=\frac{1}{q_1}+\frac{1}{q_1}\sigma(x).
$$

By analogy, we get
$$
\sigma^n(x)=\sum^{\infty} _{k=n+1}{ \frac{1}{q_{n+1}q_{n+2}\cdots q_k}}
$$
and
$$
\sigma^n(x)=\frac{1}{q_{n+1}}+\frac{1}{q_{n+1}}\sigma^{n+1}(x),
$$
where $\sigma^0(x)=x$ and $n=0, 1, 2, 3, \dots$.

Considering the differential, let us note the following auxiliary property:
\begin{equation}
\label{eq: Engel series2}
d\left(\sigma^n(x)\right)=\frac{d\left(\sigma^{n+1}(x)\right)}{q_{n+1}}.
\end{equation}

To model a correct function of the Minkowski type, let us determine Engel series of the form:
\begin{equation}
\label{eq: Engel series3}
\sum^{\infty}_{k=1}{\prod^{k} _{m=1}{\frac{1}{2+a_1+a_2+\dots +a_m-m}}},
\end{equation}
where $(a_k)$ is a fixed sequence of posaitive integers. 

Really, for $\tau_{k+1}:={2+a_1+a_2+\dots +a_m+a_{m+1}-(m+1)}$ and $\tau_{k}:=2+a_1+a_2+\dots +a_m-m$, we obtain $\tau_{k+1}-\tau_{k}\ge 0$ and $\tau_{k}\ge 2$, i.e., $2 \le\tau_{k}\le \tau_{k+1}$ holds for any $k\in \mathbb N$. 

Let us reconstruct inequalities \eqref{eq: Engel series1} into the following auxiliary inequalities:
\begin{equation}
\label{eq: Engel series4}
\begin{split}
 \prod^{k} _{m=1}{\frac{1}{2+a_1+\dots+a_m-m}} &\le x- \sum^{k-1} _{l=1}{\prod^{l} _{j=1}{\frac{1}{2+a_1+\dots+a_j-j}}}\\ 
&<\frac{1}{1+a_1+\dots +a_k-k} \prod^{k-1} _{t=1}{\frac{1}{2+a_1+\dots+a_t-t}} .
\end{split}
\end{equation}

So, let us model \emph{the Engel--Minkowski question-mark function} by the following form:
\begin{equation}
\label{eq: Engel-Minkowski}
?_{EM}\left(x\right)=2^{1-a_1}-2^{1-a_1-a_2}+\dots + (-1)^{k+1}2^{1-a_1-a_2-\dots -a_k}+\dots,
\end{equation}
where $x$ represented in terms of expansion \eqref{eq: Engel series3}, i.e., $a_k\in\mathbb N$.

\section{The Engel--Minkowski question-mark function}

\begin{theorem}
In the segment $(0,1)$, the function $?_{EM}$ is:

\begin{itemize}
\item a continuous function;
\item is a  nowhere monotonic function;
\item  a  function, which derivative at $x_0\in(0,1)$ is equal to $0$ or $\infty$ and can be calculated by the formula:
$$
\lim_{k\to\infty}{\prod^{k-1} _{r=1}{\frac{2+a_1+\dots +a_k-k}{2^{a_k}}}}.
$$
\end{itemize}
\end{theorem}
\begin{proof}
Suppose $x_1, x_2 \in (0,1)$ such that $x_1<x_2$ and
$$
x_1=\sum^{\infty}_{k=1}{\prod^{k} _{m=1}{\frac{1}{2+a_1(x_1)+a_2(x_1)+\dots +a_m(x_1)-m}}},
$$
$$
x_2=\sum^{\infty}_{k=1}{\prod^{k} _{m=1}{\frac{1}{2+a_1(x_2)+a_2(x_2)+\dots +a_m(x_2)-m}}}.
$$
Then there exists positive integer $p$  such that $a_p(x_2)< a_p(x_1)$ and $a_i(x_1)=a_i(x_2)$ for $i=\overline{1, p-1}$. 

Hence $?_{EM}(x_2) - ?_{EM}(x_1)=$
$$
=(-1)^{m+1}2^{1-a_1(x_2)-\dots - a_{p-1}(x_2) -a_p(x_2)}+(-1)^{m+2}2^{1-a_1(x_2)-\dots - a_{p-1}(x_2) -a_p(x_2)-a_{p+1}(x_2)}+\dots
$$
$$
-\left((-1)^{m+1}2^{1-a_1(x_1)-\dots - a_{p-1}(x_1) -a_p(x_1)}+(-1)^{m+2}2^{1-a_1(x_1)-\dots - a_{p-1}(x_1) -a_p(x_2)-a_{p+1}(x_1)}+\dots\right)
$$
$$
=(-1)^{m+1}2^{1-a_1-\dots - a_{p-1}}\left(\frac{1}{2^{a_p(x_2)}}-\frac{1}{2^{a_p(x_1)}}-\frac{1}{2^{a_p(x_2)}}?_{EM}\left(\sigma^p(x_2)\right)+\frac{1}{2^{a_p(x_1)}}?_{EM}\left(\sigma^p(x_1)\right)\right),
$$
since
$$
\sigma^n(x)=\sum^{\infty}_{k=1}{\prod^{k} _{r=1}{\frac{1}{2+a_{n+1}+a_{n+2}+\dots +a_{n+r}-r}}}
$$
and
$$
?_{EM}\left(\sigma^n(x)\right)=2^{1-a_{n+1}}-2^{1-a_{n+1}-a_{n+2}}+\dots
$$
are true.

So, the function is \emph{a nowhere monotonic function}, because:
\begin{itemize}
\item for the case when $p$ is odd, we have $?_{EM}(x_2) - ?_{EM}(x_1)>0$ since $a_p(x_2)< a_p(x_1)$ and $(-1)^{m+1}=1$;
\item for the case when $p$ is even, we obtain $?_{EM}(x_2) - ?_{EM}(x_1)<0$ since  $(-1)^{m+1}=-1$.
\end{itemize}
In addition, since $0<?_{EM}(x), ?_{EM}(\sigma^n(x))\le 1$ hold, we get
\begin{equation}
\label{eq: Engel series7}
\left |?_{EM}(x_2) - ?_{EM}(x_1) \right|< 2^{1-a_1-a_2-\dots - a_{p-1}-a_p(x_2)}
\end{equation}
and
$$
\left |?_{EM}(x_2) - ?_{EM}(x_1) \right|>-2^{1-a_1-a_2-\dots - a_{p-1}-a_p(x_1)}.
$$
Whence the inequalities
$$
-2^{1-a_1-a_2-\dots - a_{p-1}-a_p(x_1)}<?_{EM}(x_2) - ?_{EM}(x_1)<2^{1-a_1-a_2-\dots - a_{p-1}-a_p(x_2)}<2^{-a_1-a_2-\dots - a_{p-1}}<2^{-(p+1)}
$$
hold for the case when $p$ is odd, and
$$
-2^{1-a_1-a_2-\dots - a_{p-1}-a_p(x_2)}<?_{EM}(x_2) - ?_{EM}(x_1)<2^{1-a_1-a_2-\dots - a_{p-1}-a_p(x_1)}<2^{-a_1-a_2-\dots - a_{p-1}}<2^{-(p+1)}
$$
hold for the case when $p$ is even.

Let us prove that $?_{EM}(x)$ is continuous. Really, 
$$
\lim_{x_2-x_1\to 0}{\left|?_{EM}(x_2) - ?_{EM}(x_1) \right|}=\lim_{p\to\infty}{\left(  2^{1-a_1-a_2-\dots - a_{p-1}}\left(?_{EM}(\sigma^{p-1}(x_2))-?_{EM}(\sigma^{p-1}(x_1))\right) \right)}
$$
$$
=\lim_{p\to\infty}{\frac{1}{2^{a_1+a_2+\dots + a_{p-1}}}}=0,
$$
since $?_{EM}, \sigma^{n}$ are bounded functions. 

Let us consider differential properties of $?_{EM}$, i.e., consider the limit
$$
\lim_{x_2-x_1\to 0}{\frac{?_{EM}(x_2) - ?_{EM}(x_1) }{x_2-x_1}}.
$$
Let us begin with a certain sequence of arguments. Since for $x_1$ and $x_2$ such that there exists positive integer $p$  for which $a_p(x_2)< a_p(x_1)$ and $a_i(x_1)=a_i(x_2)$ for $i\ne p$,the following equalities hold: 
$$
x_2-x_1=\frac{1}{2+a_1+a_2+\dots +a_{p-1}+a_p(x_2)-p}\prod^{p-1} _{r=1}{\frac{1}{2+a_1+a_2+\dots +a_r-r}}
$$
$$
-\frac{1}{2+a_1+a_2+\dots +a_{p-1}+a_p(x_1)-p}\prod^{p-1} _{r=1}{\frac{1}{2+a_1+a_2+\dots +a_r-r}}+
$$
$$
+\prod^{p-1} _{r=1}{\frac{1}{2+a_1+\dots +a_r-r}}\left(\frac{\sigma^p(x_{1,2})}{2+a_1+\dots +a_{p-1}+a_p(x_2)-p}-\frac{\sigma^p(x_{1,2})}{2+a_1+\dots +a_{p-1}+a_p(x_1)-p}\right)
$$
$$
=(1+\sigma^p(x_{1,2}))\left(\frac{1}{2+a_1+\dots +a_{p-1}+a_p(x_2)-p}-\frac{1}{2+a_1+\dots +a_{p-1}+a_p(x_1)-p}\right)\times
$$
$$
\times\prod^{p-1} _{r=1}{\frac{1}{2+a_1+\dots +a_r-r}}
$$
and
$$
?_{EM}(x_2) - ?_{EM}(x_1)=(-1)^{p+1}(1-?_{EM}(\sigma^p(x_{1,2})))2^{1-a_1+\dots+a_{p-1}}\left(\frac{1}{2^{a_p(x_2)}}-\frac{1}{2^{a_p(x_1)}}\right)
$$
we obtain
$$
\lim_{x_2-x_1\to 0}{\frac{?_{EM}(x_2) - ?_{EM}(x_1) }{x_2-x_1}}=\lim_{p\to\infty}{\frac{\prod^{p-1} _{r=1}{{(2+a_1+\dots +a_r-r)}}}{2^{a_1+a_2+\dots +a_{p-1}}}}
$$
$$
=\lim_{p\to\infty}{\prod^{p-1} _{r=1}{\frac{2+a_1+\dots +a_r-r}{2^{a_r}}}}.
$$

Whence, if $(a_k)$ is a sequence such that for a finite subsequence of which the condition
$$
\frac{2+a_1+\dots +a_r-r}{2^{a_r}}>1
$$
is true, then the derivative equals zero. So, we obtain that the derivative of $?_{EM}$ can be equal   $0$ or $\infty$.

Let note the other example.  For this, let use the auxiliary notion of cylinders (on these sets see, for example, in \cite{ShW2020}). 

\emph{A cylinder $I(c_1c_2\ldots c_n)$ of rank $n$ with base $c_1c_2\ldots c_n$} is a set, whose elements have fixed digits $q_i=c_i$, $i=\overline{1,n}$, in own Engel expansion. This is an is an interval with following two endpoints (\cite{ShW2020}):
$$
\sum^{n} _{k=1}{\prod^{k}_{l=1}{\frac{1}{c_l}}}
$$
and
$$
\frac{1}{c_1c_2\cdots c_{n-1}(c_n-1)}+\sum^{n-1} _{k=1}{\prod^{k}_{l=1}{\frac{1}{c_l}}}
$$
with the Lebesque measure, which equals
$$
\frac{1}{c_1c_2\cdots c_{n-1}c_n(c_n-1)}.
$$
Then the derivative on cylinders can be calculated by the following:
$$
\lim_{n\to \infty}{\frac{2^{1-c_1-c_2-\dots - c_{n-1}}\left(\frac{1}{2^{c_n-1}}-\frac{1}{2^{c_n}}\right)}{\frac{1}{\tau_1\tau_2\cdots \tau_{n-1}\tau_n (\tau_n-1)}}}=\lim_{n\to \infty}{\frac{\tau_1\tau_2\cdots \tau_{n-1}\tau_n}{2^{c_1+c_2+\dots +c_n}}},
$$
where $\tau_i=2+c_1+c_2+\dots+c_i-i$ for $i=\overline{1,n}$.
\end{proof}

\begin{theorem}  The following system of functional equations
\begin{equation}
\label{eq: system1}
{f\left(\sigma^{n-1}(x)\right)}=\frac{1}{2^{a_{n}}}- \frac{1}{2^{a_{n}}}f\left(\sigma^n(x)\right),
\end{equation}
where $x$ represented by Engel series \eqref{eq: Engel series3}, $a_n\in \mathbb N$, $n=1,2, \dots$, $\sigma$ is the shift operator, and $\sigma_0(x)=x$, has the unique solution
$$
\frac{?_{EM}(x)}{2}=2^{-a_1}-2^{-a_1-a_2}+\dots + (-1)^{n+1}2^{-a_1-a_2-\dots -a_n}+\dots
$$
in the class of determined and bounded on $(0, 1]$ functions.
\end{theorem}
\begin{proof} Our statement is true, because the following properties and relationships hold: the function  is a determined and bounded  function on $(0,1]$, as well as according to system~\eqref{eq: system1},  we have
$$
\frac{?_{EM}(x)}{2}=\frac{1}{2^{a_{1}}}- \frac{1}{2^{a_{1}}}f(\sigma(x))=\frac{1}{2^{a_{1}}}- \frac{1}{2^{a_{1}}}\left(\frac{1}{2^{a_{2}}}- \frac{1}{2^{a_{2}}}f(\sigma^2(x))\right)
$$
$$
\dots =2^{-a_1}-2^{-a_1-a_2}+\dots + (-1)^{n+1}2^{-a_1-a_2-\dots -a_n}+\frac{(-1)^{n+2}}{2^{a_1+a_2+\dots +a_n}}f(\sigma^n(x)).
$$
This statement is true, because
$$
\lim_{n\to\infty}{\frac{(-1)^{n+2}}{2^{a_1+a_2+\dots +a_n}}f(\sigma^n(x))}=0.
$$
\end{proof}

\begin{theorem}
For the Lebesgue integral, the following equality holds:
$$
\int^1 _0 {?_{EM}(x)dx}=2\frac{\sum^{\infty} _{a=1}{\frac{1}{a(a+1)2^a}}}{1+\sum^{\infty} _{a=1}{\frac{1}{a(a+1)^22^a}}}.
$$
\end{theorem}
\begin{proof} Using system \eqref{eq: system1}, equality \eqref{eq: Engel series2}, and according to self-affine properties of $?_{EM}$, we obtain
$$
d(\sigma^n(x))=\frac{d(\sigma^{n+1}(x))}{2+a_1+a_2+\dots +a_{n+1}-n-1},
$$
as well as
$$
?_{EM}(\sigma^n(x))=\frac{1}{2^{a_{n}}}\left(1-?_{EM}(\sigma^{n+1}(x))\right).
$$
So, 
$$
\frac 1 2 \int^1 _0 {?_{EM}(x)dx}=\int^1 _0{ \frac{1}{2^{a_{1}}}\left(1-?_{EM}(\sigma(x))\right)}dx
$$
$$
=\int^1 _0 { \frac{1}{2^{a_{1}}}}dx-\int^1 _0 { \frac{1}{2^{a_{1}}}\cdot \frac{1}{1+a_1}?_{EM}(\sigma(x))}d(\sigma(x))
$$
$$
=\sum^{\infty} _{a_1=1}{\int^{\frac{1}{a_1}} _{\frac{1}{a_1+1}}{ \frac{1}{2^{a_{1}}}}}dx-\sum^{\infty} _{a_1=1}{\int^{\frac{1}{a_1}} _{\frac{1}{a_1+1}}{ \frac{1}{2^{a_{1}}}\cdot \frac{1}{1+a_1}?_{EM}(\sigma(x))}d(\sigma(x))}
$$
$$
=\sum^{\infty} _{a=1}{\frac{1}{a(a+1)2^a}}-\left(\sum^{\infty} _{a=1}{\frac{1}{a(a+1)^22^a}}\right)\int^1 _0{?_{EM}(\sigma(x))d(\sigma(x))}.
$$

Since for all steps $a_n = 1, 2, 3, \dots$ and according to self-affine properties of this function, we have
$$
\int^1 _0 {?_{EM}(x)dx}=2\frac{\sum^{\infty} _{a=1}{\frac{1}{a(a+1)2^a}}}{1+\sum^{\infty} _{a=1}{\frac{1}{a(a+1)^22^a}}}.
$$
\end{proof}

\section*{Statements and Declarations}
\begin{center}
{\bf{Competing Interests}}

\emph{The author states that there is no conflict of interest.}
\end{center}

\begin{center}
{\bf{Data Availability Statement}}

\emph{There are not suitable for this research.}
\end{center}

\end{document}